\documentclass[11pt]{article}
\usepackage{amsmath, amsfonts, amsthm}
\newtheorem{theo}{Theorem}[section]

\newtheorem{lemma}[theo]{Lemma}
\newtheorem{coro}[theo]{Corollary}
\newtheorem{conj}[theo]{Conjecture}

\newcommand{\GG}{{\cal G}}

\newcommand{\eps}{{\varepsilon}}

\def \RR {R}

\def \PP {\mathbb P}

\begin{document}
\date{}

\title{
Optimal compression of approximate inner products and
dimension reduction
}

\author{Noga Alon\textsuperscript{1}
\and
Bo'az Klartag\textsuperscript{2}
}

\footnotetext[1]{Sackler School of Mathematics
and Blavatnik School of
Computer Science, Tel Aviv University, Tel Aviv 69978, Israel.
Email: {\tt nogaa@tau.ac.il}.  Research supported in part by a
USA-Israeli
BSF grant 2012/107, by an ISF grant 620/13 and
by the Israeli I-Core program.}
\footnotetext[2]{
Sackler School of
Mathematics, Tel Aviv University, Tel Aviv 69978, Israel
and Department of Mathematics,
Weizmann Institute of Science, Rehovot 7610001, Israel.
Email: {\tt klartagb@tau.ac.il}.  Research supported in part by an
ERC  grant.
}

\maketitle

\begin{abstract}

Let $X$ be a set of $n$ points of norm at most $1$ in the Euclidean
space $R^k$, and suppose $\eps>0$. An $\eps$-distance sketch for $X$ is a
data structure that, given  any two points of $X$ enables one to recover
the square of the
(Euclidean) distance between them up to an {\em additive} error of
$\eps$. Let $f(n,k,\eps)$ denote the minimum possible number of bits of
such a sketch.  Here we determine $f(n,k,\eps)$ up to a constant factor
for all $n \geq k \geq 1$ and all $\eps \geq \frac{1}{n^{0.49}}$. Our
proof is algorithmic, and provides an efficient algorithm for computing
a sketch of size $O(f(n,k,\eps)/n)$ for each point, so that the
square of the distance
between any two points can be computed from their sketches up to an
additive error of $\eps$ in time linear in the length of the sketches.
We also discuss the case of smaller  $\eps>2/\sqrt n$ and
obtain some new results about dimension reduction in this
range. In particular, we show that for any such $\eps$ and any
$k \leq t=\frac{\log (2+\eps^2 n)}{\eps^2}$ there are
configurations of $n$ points in $R^k$ that cannot be embedded in
$R^{\ell}$ for $\ell < ck$ with $c$ a small absolute positive
constant, without distorting some inner products (and distances)
by more than
$\eps$. On the positive side, we provide a randomized
polynomial time algorithm for a bipartite variant
of the Johnson-Lindenstrauss lemma in which scalar products
are approximated up to an additive error of at most $\eps$.
This variant allows a reduction of the dimension down to
 $O(\frac{\log (2+\eps^2 n)}{\eps^2})$, where $n$ is the number of
points.

\end{abstract}

\newpage
\setcounter{page}{1}

\section{Introduction}
A crucial tool in several important algorithms is the ability to
generate a compact representation (often called a sketch)
of high dimensional data.
Examples include streaming  algorithms \cite{AMS}, \cite{Mu},
compressed sensing \cite{CRT} and  data structures supporting
nearest neighbors search \cite{AC}, \cite{HIM}.
A natural problem in this area
is that of representing a collection of $n$ points in the $k$-dimensional Euclidean ball in a way that enables one to
recover approximately the distances or the inner products between
the points. The most basic question about it is the minimum
possible number of bits required in such a representation as a
function of $n,k$ and the approximation required. Another challenge
is to design economic sketches that can be generated efficiently
and support efficient  procedures for recovering the approximate
inner product (or distance) between any two given points.

Consider a sketch that enables one to
recover each inner product (or square distance) between any pair of
the $n$ points up to an additive error of $\eps$.
The Johnson-Lindenstrauss Lemma \cite{JL} provides an elegant
way to generate such a sketch.
The assertion of the lemma is that any set of $n$ points
in an Euclidean space can be projected onto a $t$-dimensional
Euclidean space, where $t=\Theta(\frac{\log n}{\eps^2})$, so that all
distances and inner products between pairs of points
are preserved up to a factor of
$1+\eps$. This supplies a sketch obtained by storing the
(approximate) coordinates of the projected points. Although the
above  estimate for $t$ has been recently shown by
Larsen and Nelson \cite{LN} to be tight up to a constant factor
for all $\eps \geq \frac{1}{n^{0.49}}$, improving by a logarithmic
factor the estimate in \cite{Al}, this does not provide a
tight estimate for the minimum possible number of bits required for
the sketch.
The results in
Kushilevitz, Ostrovsky and Rabani \cite{KOR} together with the lower
bound in \cite{LN}, however,
determine the minimum possible number of bits
required for such a sketch up to a
constant factor for all $k \geq \frac{\log n}{\eps^2}$ where
$\eps \geq \frac{1}{n^{0.49}}$, leaving a gap in the bounds for
smaller dimension $k$. Our first result here closes this gap.

\subsection{Our contribution}
Let $X$ be a set of $n$ points of norm at most $1$ in the Euclidean space
$R^k$, and suppose $\eps>0$. An $\eps$-distance sketch for $X$ is a data
structure that, given  any two points of $X$ enables one to recover the
square of the Euclidean distance between them, and their inner product,
up to an {\em additive} error of $\eps$. Let $f(n,k,\eps)$ denote the
minimum possible number of bits of such a sketch.  Our first main result
is a determination of $f(n,k,\eps)$ up to a constant factor for all $n
\geq k \geq 1$ and all $\eps \geq \frac{1}{n^{0.49}}$.
\begin{theo}
\label{t11}
For all $n$ and $\frac{1}{n^{0.49}} \leq \eps \leq 0.1$ the
function $f(n,k,\eps)$ satisfies the following
\begin{itemize}
\item
For
$ \frac{\log n}{\eps^2}  \leq k \leq n$,
$$
f(n,k,\eps)
=\Theta(\frac{n \log n}{\eps^2}).
$$
\item
For $\log n \leq k \leq \frac{\log n}{\eps^2}$,
$$
f(n,k,\eps)=\Theta (nk \log (2+\frac{\log n}{\eps^2 k})).
$$
\item
For  $1 \leq k \leq \log n$,
$$
f(n,k,\eps)=\Theta(nk \log (1/\eps)).
$$
\end{itemize}
\end{theo}

The proof is
algorithmic, and provides an efficient algorithm for computing a sketch
of size $O(f(n,k,\eps)/n)$ for each point, so that the square of the
distance between any two points can be computed from their sketches up to
an additive error of $\eps$ in time linear in the length of the sketches.
The tight bounds show that if $\eps\geq \frac{1}{n^{0.49}}$ and $\ell
\leq c \frac{\log n}{\eps^2}$ for some (small) absolute positive constant
$c$, then
$f(n,k,\eps)$ for $k=\frac{\log n}{\eps^2}$ is significantly larger than
$f(n,\ell,2\eps)$, supplying an alternative proof of the main result of
\cite{LN} which shows that the $\frac{\log n}{\eps^2}$ estimate in the
Johnson-Lindenstrauss dimension reduction lemma is tight.

An advantage of this alternative
proof is that an appropriate adaptation of it
works for smaller values of $\eps$, covering all the relevant
range. For any $\eps \geq \frac{2}{\sqrt n}$, define $t=\frac{\log
(2+\eps^2 n)}{\eps^2}$. We show that for every $k \leq t$ there is
a collection of $n$ points of norm at most $1$ in $R^k$, so that
in any embedding of them in dimension $\ell$ such that no
inner product (or distance) between a pair of points
is distorted by more than $\eps$, the dimension $\ell$ must be at
least $\Omega(k)$. This extends the main result of \cite{LN}, where
the above is proved only for
$\eps \geq \frac{\log^{0.5001} n}{\sqrt k}$.

The above result for small values of $\eps$ suggests that it may be
possible to improve the Johnson-Lindenstrauss Lemma in this range.
Indeed, our
second main result addresses dimension reduction in this range.
Larsen and Nelson \cite{LN1} asked if for any
$\eps$ the assertion of the
Johnson-Lindenstrauss Lemma can be improved, replacing
$\frac{\log n}{\eps^2}$ by $t=\frac{\log(2+\eps^2 n)}{\eps^2}$.
(Note that this is trivial for $\eps<\frac{1}{\sqrt n}$ as
in this range $t>n$ and it is true for $\eps>\frac{1}{n^{0.49}}$
as in this case $\log(2+\eps^2 n)=\Theta(\log n)$.)
Motivated by this we prove the following bipartite version of this
statement.
\begin{theo}
\label{t13}
There exists an absolute positive constant $C$ such that for every
vectors $a_1,a_2, \ldots ,a_n,b_1, b_2, \ldots ,b_n \in R^n$, each
of Euclidean norm at most $1$, and for every $0<\eps<1$ and
$t=\lfloor C \frac{\log (2+\eps^2 n)}{\eps^2} \rfloor$ there are vectors
$x_1,x_2, \ldots ,x_n, y_1, y_2, \ldots, y_n \in R^t$ so that
for all $i,j$
$$
|\langle x_i, y_j \rangle-\langle a_i, b_j \rangle | \leq \eps
$$
\end{theo}
The proof of the theorem is algorithmic, providing a randomized
polynomial time algorithm for computing the vectors $x_i,y_j$ given
the vectors $a_i,b_j$.

\subsection{Related work}

As mentioned above, one way to obtain a sketch for the above
problem when $k \geq \frac{\log n}{\eps^2}$ and $\eps \geq
\frac{1}{n^{0.49}}$ is to apply the
Johnson-Lindenstrauss Lemma \cite{JL} (see \cite{AC} for an
efficient implementation) projecting the points into
a $t$-dimensional space, where $t=\Theta(\frac{\log n}{\eps^2})$, and then
rounding each point to its closest neighbor in an appropriate
$\eps$-net. This provides a sketch of size $O(t \log (1/\eps))$ bits
per point, which by the results in \cite{LN} is optimal up to a
$\log(1/\eps)$ factor for these values of $n$ and $k$.

A tight upper bound of $O(t)$ bits per point for these values of
the parameters, with an efficient recovery procedure,
follows from the work of \cite{KOR}. Their work does not seem to
provide tight bounds for smaller values of $k$.

A very recent paper of Indyk and Wagner \cite{IW} addresses the
harder problem of approximating the inner products between pairs of
points up to a {\em relative} error of $\eps$, for the special case
$k=n$, and determines the minimum number of bits required here
up to a factor of $\log(1/\eps)$.

There have been  several papers dealing with the tightness of the
dimension $t$ in the Johnson-Lindenstrauss lemma, culminating with
the recent work of Larsen and Nelson that determines it
up to a constant factor for $\eps \geq \frac{1}{n^{0.49}}$
(see \cite{LN} and the references therein). For smaller values of
$\eps$ the situation is more complicated. Our results here,
extending the one of \cite{LN}, show
that no reduction to dimension smaller than
$t=\frac{\log (2+\eps^2 n)}{\eps^2}$ is possible, for
any $n \geq k \geq t$ and any $\eps>\frac{2}{\sqrt n}$. (For any
smaller value of $\eps$, or for any $k<t$
no reduction by more than a constant factor
is possible).
There is no known improvement in the statement of the
Johnson-Lindenstrauss Lemma for small values of $\eps$,
and our bipartite
version and some related results proved here are the first to
suggest that such an improvement may indeed hold.

\subsection{Techniques}

Our arguments combine probabilistic and geometric tools. The
lower bound for the function $f(n,k,\eps)$ is proved by a probabilistic
argument. We provide two proofs of the upper bound.  The first is based
on a short yet intriguing volume argument. Its main disadvantage
is that it
is not constructive, and its main advantage is that by combining
it with results about Gaussian correlation it can be extended to
deal with
smaller values of $\eps$ as well, for all the relevant range.
The second proof is algorithmic
and is based on randomized rounding.

The results about improved (bipartite) dimension reduction for small
$\eps$ are proven using several tools  from convex geometry including
the low-$M^*$ estimate and the finite volume-ratio theorem (see, e.g.,
\cite{AGM}), and basic results about the positive correlation between
symmetric convex events with the Gaussian measure. We believe that
these tools may be useful in the study of related algorithmic
questions in high dimensional geometry.

\section{Formal statement of the results}

Theorem \ref{t11} supplies an alternative proof of the main result of
\cite{LN} about dimension reduction. For $n \geq k \geq \ell$
and $\eps \geq \frac{1}{n^{0.49}}$ we
say that there is an $(n,k,\ell, \eps)$-Euclidean dimension reduction
if for any points $x_1,\ldots,x_n \in R^k$ of norm at most one,
there exist points $y_1,\ldots,y_n \in R^{\ell}$ satisfying
\begin{equation} |x_i - x_j|^2 - \eps \leq
 |y_i - y_j|^2 \leq |x_i - x_j|^2 + \eps \qquad \qquad (i,j=1,\ldots,n).
\label{eq_1005}
 \end{equation}
\begin{coro}
\label{c12}
There exists an absolute positive constant $c>0$ so that for
any $n \geq k>ck \geq \ell$ and for $1/n^{0.49} \leq \eps \leq 0.1$,
there is an $(n, k, \ell, \eps)$-Euclidean dimension reduction
if and only if $\ell = \Omega(\log n / \eps^2)$.

\medskip Moreover, the same holds if we replace additive distortion
by multiplicative distortion, i.e., if we
replace  condition (\ref{eq_1005}) by the following condition
\begin{equation} (1 - \eps) \cdot |x_i - x_j|^2 \leq
|y_i - y_j|^2 \leq (1 + \eps) \cdot |x_i - x_j|^2
\qquad \qquad (i,j=1,\ldots,n). \label{eq_1015}
 \end{equation}
\end{coro}

Corollary \ref{c12} means that if $k \geq c_1 \log n / \eps^2$, then
there is an $(n, k, \eps^{-2} \log n,\eps)$-Euclidean dimension reduction
(by the Johnson-Lindenstrauss Lemma),
and that if there is an $(n, k, \ell,\eps)$-Euclidean dimension reduction
with $\ell=o(k)$ then necessarily $k \geq \ell \geq c_2 \eps^{-2} \log n $,
for some absolute constants $c_1, c_2 > 0$.

In Theorem \ref{t11} and Corollary \ref{c12} it is assumed that  $\eps \geq \frac{1}{n^{0.49}}$.
For smaller $\eps$ we can combine some of our techniques with
Harg\'e's Inequality about Gaussian correlation and prove the
following extension of Theorem \ref{t11}.
\begin{theo}
\label{t23}
For all $n$ and $\eps  \geq \frac{2}{\sqrt{n}}$ the
function $f(n,k,\eps)$ satisfies the following,
where $t=\frac{\log (2+\eps^2 n)}{\eps^2}$.
\begin{itemize}
\item
For
$ t  \leq k \leq n$,
$$
\Omega(nt) \leq f(n,k,\eps) \leq O(n \frac{\log n}{\eps^2}).
$$
\item
For $\log (2+\eps^2 n) \leq k \leq t$,
$$
f(n,k,\eps)=\Theta (nk \log (2+\frac{t}{k})).
$$
\item
For  $1 \leq k \leq \log (2+\eps^2 n)$,
$$
f(n,k,\eps)=\Theta(nk \log (1/\eps)).
$$
\end{itemize}
\end{theo}
This implies the following result about dimension
reduction.
\begin{coro}
\label{c24}
There exists an absolute positive constant $c>0$ so that for
any $n \geq k>ck \geq \ell$ and for all $\eps \geq\frac{2}{\sqrt n}$,
if there is an
$(n, k, \ell, \eps)$-Euclidean dimension reduction
then $\ell = \Omega(\frac{\log (2+\eps^2 n)}{\eps^2}$.
\end{coro}

Note that for the range of $\eps$ in which
$\log(2+\eps^2 n) =o(\log n)$ the statements of Theorem
\ref{t23} and of Corollary \ref{c24} are essentially the ones
obtained from those in Theorem \ref{t11} and Corollary \ref{c12}
by replacing the
term $\log n/{\eps^2}$
by the expression
$t=\frac{\log (2+\eps^2 n)}{\eps^2}$.
In fact, it is possible that as suggested
by Larsen and Nelson \cite{LN1} for such
small values of $\eps$ the assertion of the
Johnson-Lindenstrauss Lemma can also be improved, replacing
$\frac{\log n}{\eps^2}$ by $\frac{\log(2+\eps^2 n)}{\eps^2}$.
Motivated by this we prove a bipartite version of the result,
stated as Theorem \ref{t13} in the previous section.
We conjecture that the assertion of this  theorem can be
strengthened, as follows.
\begin{conj}
\label{c14}
Under the assumptions of Theorem \ref{t13}, the conclusion holds
together with the further requirement that
$ \|x_i\| \leq O(1)$ and $\|y_i \| \leq O(1)$ for all $1 \leq i
\leq n$.
\end{conj}

Note that the assertion of the conjecture is
trivial for $\eps < \sqrt {C / (2n)}$, as in that case $t \geq n$.
Note also that for, say, $\eps>1/n^{0.49}$  the assertion holds
by the Johnson-Lindenstrauss Lemma.

If true, this conjecture, together with our methods here, suffices to
establish a
tight upper bound up to a constant factor  for the number of bits
required for maintaining all inner products between $n$ vectors
of norm at most $1$ in $R^n$, up to an additive error of $\eps$
in each product, for all $\eps \geq \frac{2}{\sqrt n}$,
closing the gap between the upper and lower bound in the first
bullet in Theorem
\ref{t23}.
The conjecture, however, remains open,
but we can
establish two results supporting it. The first is a proof of the
conjecture when $t$ is $n/2$ (or more generally $\Omega(n)$,
that is, the case $\eps=\Theta(1/\sqrt n)$). Our result is as follows:

\begin{theo}
Let $m \geq n \geq 1, \eps > 0$ and assume that
$a_1,\ldots,a_m, b_1,\ldots,b_m \in \RR^{2n}$
are points of norm at most one.
Suppose that $X_1,\ldots,X_m, Y_1,\ldots,Y_m \in \RR^n$
are independent random vectors, distributed according to
standard Gaussian law. Set $\bar{X}_i = X_i /
\sqrt{n}$ and $\bar{Y}_i = Y_i / \sqrt{n}$ for all $i$.

\medskip Assume that
$n \geq C_1 \frac{\log (2+\eps^2 m)}{\eps^2}$.
Then with probability of at least $\exp(-C_2 n m)$,
$$
 \left| \left \langle \bar{X}_i, \bar{Y}_j
\right \rangle - \left \langle a_i, b_j \right \rangle \right |
\leq \eps   \qquad \qquad \text{for} \ i,j=1,\ldots,m,
$$
and moreover $|\bar{X}_i| + |\bar{Y}_i| \leq C_3$
for all $i$. Here, $C_1, C_2, C_3 > 0$ are universal constants.
\label{thm_944}
\end{theo}

The second result is an estimate, up to a constant factor, of the number
of bits required to represent, for a given set of $n$ vectors
$a_1,a_2, \ldots ,a_n \in \RR^k$, each of
norm at most $1$, the sequence of all inner products
$\langle a_i, y \rangle$ with a vector $y$ of norm at most $1$ in $R^k$
up to an additive error of $\eps$ in each such product. This
estimate is the same, up to a constant factor, for all
dimensions $k$ with $t \leq k \leq n$ and $t$ as in Theorem
\ref{t13}, as should
be expected from the assertion of the Conjecture.

The remainder of this paper is structured as follows.
In Section 3 we provide our first proof of the upper bound in Theorem \ref{t11},
which is  based on a short probabilistic (or volume) argument.
The second proof, presented in Section 4,
is algorithmic. It provides  an efficient randomized algorithm for
computing a sketch consisting of $O(f(n,k,\eps)/n)$ bits for each point
of $X$, so that the square of the
distance between any two points can be
recovered, up to an additive error of $\eps$, from their sketches,
in time linear in the length of the sketches.
Section 5 is concerned with the lower bound in Theorem \ref{t11}.
The results on smaller $\eps$ are proven in Section 6
using several tools  from convex geometry.
The final section 7
contains some concluding remarks and open problems.

Throughout the proofs we make
no serious attempt to optimize the absolute constants involved.
We write $c, \tilde{C}, c_1,\ldots$ etc.
for various positive universal constants, whose
values may change from one line to the next.
We usually use upper-case $C$ to denote universal
constants that we consider ``sufficiently large'',
and lower-case $c$ to denote universal constants
that are sufficiently small.
For convenience we sometimes bound $f(n,k,2 \eps)$ or
$f(n,k,5\eps)$
instead of $f(n,k,\eps)$, the corresponding bounds for
$f(n,k,\eps)$ follow, of course, by replacing $\eps$ by
$\eps/2$ or $\eps/5$ in the expressions we get, changing the
estimates only by a constant factor. All logarithms are in the
natural basis $e$ unless otherwise specified.

\section{The upper bound}
It is convenient to split the proof of the upper bound in Theorem
\ref{t11} into three lemmas, dealing with
the different ranges of $k$. The proof of the upper bound in
Theorem \ref{t23}, presented in Section 6, combines a similar
reasoning with results of Khatri, Sidak \cite{Kh}, \cite{Si} and
Harg\'e \cite{Ha} about the Gaussian correlation Inequality.
\begin{lemma}
\label{l21}
For
$ \frac{\log n}{\eps^2}  \leq k \leq n$,
$$
f(n,k,5 \eps)
\leq O(\frac{n \log n}{\eps^2}).
$$
\end{lemma}
\vspace{0.1cm}

\noindent
{\bf Proof:}\, Since $f(n,k,5 \eps)$ is clearly a monotone
increasing function  of $k$, it suffices to prove the upper bound
for $k=n$. By \cite{JL} we can replace the points of $X \subset
B^k$, where $B^k$ is the unit ball in $R^k$,
by points
in $R^m$ where $m=C \frac{\log n}{\eps^2}$ so that all distances
and norms of the points change by at most $\eps$. Hence we may
and will assume that our set of points $X$ lies in $R^m$.
Note that given the squares of the norms
of two vectors up to an additive error of $\eps$ and
given their inner product up to an additive error of $\eps$
we get an approximation of the square of their distance up to an
additive error of $4 \eps$. It thus suffices to show the
existence of a sketch that can provide the approximate norm of
each of our vectors and the approximate inner products between
pairs. The approximate norms can be stored trivially by
$O( \log (1/\eps))$ bits per vector. (Note that here the cost for
storing even a much better approximation for the norms is
negligible, so if the constants are important we can ensure that
the norms are known with almost no error). It remains to prepare a
sketch for the inner products.

The Gram matrix $G(w_1,w_2, \ldots ,w_n)$  of $n$ vectors
$w_1,\ldots,w_n$ is the $n$ by $n$ matrix $G$ given by
$G(i,j)=\langle w_i, w_j \rangle$. We say that two Gram matrices $G_1,G_2$
are $\eps$-separated if there are two indices $i \neq j$
so that $|G_1(i,j) -G_2(i,j)| >  \eps$.  Let $\GG$ be a maximal
(with respect to containment) set of Gram matrices  of ordered
sequences of $n$ vectors $w_1, \ldots ,w_n$ in $R^m$, where
the norm of each vector $w_i$ is at most $2$, so that every
two distinct members of $\GG$ are $\eps$-separated. Note that by
the maximality of $\GG$, for every Gram matrix $M$
of $n$ vectors of norms at
most $2$ in $R^m$ there is a member of $\GG$ in which
all inner products of pairs of distinct points
are within $\eps$ of the corresponding inner products
in $M$, meaning that as a sketch for $M$ it suffices to
store (besides the approximate norms of the vectors),
the index of an appropriate member of $\GG$. This requires
$\log |\GG|$ bits. It remains to prove an upper bound for
the cardinality of $\GG$. We proceed with that.

Let $V_1,V_2, \ldots ,V_n$ be $n$ vectors, each chosen
randomly, independently and uniformly in the ball of
radius $3$ in $R^m$ centered at $0$. Let $T=G(V_1,V_2, \ldots ,V_n)$
be the Gram matrix of the vectors $V_i$.
For each $G \in \GG$ let $A_G$ denote the
event that for every $1 \leq i \neq j \leq n$,
$|T(i,j)-G(i,j)| < \eps/2$. Note that since the members of
$\GG$ are $\eps$-separated, all the events $A_G$ for $G \in \GG$
are pairwise disjoint. We claim that the probability of each
event $A_G$ is at least $0.5 (1/3)^{mn}$. Indeed,
fix a Gram matrix $G = G(w_1,\ldots,w_n) \in \GG$ for
some $w_1,\ldots,w_n \in \RR^m$
of norm at most $2$. For each fixed
$i$ the probability that
$V_i$ lies in the unit ball centered at $w_i$ is exactly
$(1/3)^m$. Therefore the probability that this happens for all $i$
is exactly $(1/3)^{nm}$. The crucial observation is that
conditioning on that, each vector $V_i$ is uniformly  distributed
in the unit ball centered at $w_i$. Therefore, after the
conditioning, for each $i \neq j$
the probability
that the inner product $\langle V_i-w_i, w_j \rangle$
has absolute value at
least
$\eps/4$ is at most $2e^{-\eps^2 m/64}<1/(2n^2)$. (Here we used
the fact that the norm of $w_j $ is at most $2$ and that the
constant $C$ in the definition of $m$ is sufficiently large).
Similarly, since
the norm of $V_i$ is at most  $3$, the
probability that the inner product $\langle V_i, V_j-w_j \rangle$
has absolute
value at least $\eps/4$ is at most $2 e^{-\eps^2 m/96}<1/2n^2$.
It follows that with probability bigger than $0.5(1/3)^{nm}$
all these inner products are smaller than $\eps/4$, implying that
$$
|\langle V_i, V_j \rangle- \langle w_i,w_j \rangle|
\leq | \langle V_i-w_i, w_j \rangle|
+|\langle V_i, V_j-w_j \rangle |<\eps/2.
$$
This proves that the probability of each event $A_G$ is at least
$0.5(1/3)^{nm}$, and as these are pairwise disjoint their number is
at most $2 \cdot 3^{nm}$,
completing the proof of the lemma.  \hfill $\Box$
\vspace{0.2cm}

\noindent
\begin{lemma}
\label{l22}
\item
For $\log n \leq k \leq \frac{\log n}{\eps^2}$,
$$
f(n,k,4 \eps) \leq O(nk \log (2+\frac{\log n}{\eps^2 k})).
$$
\end{lemma}
\vspace{0.1cm}

\noindent
{\bf Proof:}\, The proof is nearly identical to the second
part of the proof above. Note, first, that by monotonicity and the
fact that the expressions above change  only by a constant factor
when $\eps$ changes by a constant factor,
it suffices to prove the required bound for
$k=\frac{\delta^2}{\eps^2} \log n$ where $2 \eps \leq \delta
\leq 1/2$.
Let $\GG$ be a maximal set of $\eps$-separated Gram matrices of
$n$ vectors of norm at most $1$ in $R^k$. (Here it suffices to
deal with norm $1$ as we do not need to start with
the Johnson-Lindenstrauss Lemma which may slightly increase norms).
In order to prove an upper
bound for $\GG$ consider, as before, a fixed Gram matrix
$G=G(w_1, \ldots ,w_n)$ of $n$ vectors of norm at most
$1$ in $R^k$.
Let $V_1,V_2, \ldots ,V_n$ be random vectors
distributed  uniformly and independently in the ball of radius
$2$ in $R^k$, let $T$ denote their Gram matrix,
and let $A_G$ be, as before, the event that $T(i,j)$ and $G(i,j)$
differ by less than  $\eps/2$ in each non-diagonal entry.
The probability that each $V_i$ lies in the ball of
radius, say, $\delta /20$ centered at $w_i$ is exactly
$(\delta/40)^{kn}$. Conditioning on that, the probability that
the inner product $ \langle V_i-w_i, w_j \rangle$ has absolute value at
least
$\eps/4$ is at most
$$
2e^{-\eps^2 400 k/32 \delta^2}<1/(2n^2).
$$
Similarly, the
probability that the inner product $\langle V_i, V_j-w_j \rangle$
has absolute
value at least $\eps/4$ is at most
$$
2 e^{-\eps^2 400 k/64 \delta^2}<1/2n^2.
$$
As before, this implies that $|\GG| \leq 2 (40/\delta)^{kn}$,
establishing the assertion of the lemma.  \hfill $\Box$
\vspace{0.2cm}

\noindent
\begin{lemma}
\label{l23}
\item
For $k \leq \log n$,
$$
f(n,k,\eps) \leq O(nk \log (1/\eps)).
$$
\end{lemma}
\vspace{0.1cm}

\noindent
{\bf Proof:}\,
Fix an $\eps/2$-net of size $(1/\eps)^{O(k)}$
in the unit ball in $R^k$.
The sketch here is simply obtained by representing each point by
the index of its closest neighbor in the net.  \hfill $\Box$

\section{An algorithmic proof}

In this section  we present an algorithmic proof of the upper bound
of Theorem \ref{t11}. We first reformulate the theorem in its
algorithmic version. Note that the first part also follows from the
results in \cite{KOR}.
\begin{theo}
\label{t31}
For all $n$ and $\frac{1}{n^{0.49}} \leq \eps \leq 0.1$ there
is a randomized algorithm that given a set of $n$ points in
the $k$-dimensional unit ball $B^k$
computes, for each point, a sketch of $g(n,k,\eps)$ bits.
Given two sketches, the square of the
distance between the points can be recovered
up to an additive error of $\eps$ in time
$O(\frac{\log n}{\eps^2})$ for $\frac{\log n}{\eps^2} \leq k \leq
n$ and in time
$O(k)$ for all smaller $k$.
The function $g(n,k,\eps)$ satisfies the following
\begin{itemize}
\item
For
$ \frac{\log n}{\eps^2}  \leq k \leq n$,
$$
g(n,k,\eps)
=\Theta(\frac{\log n}{\eps^2})
$$
and the sketch for a given point can be computed in time
$O(k \log k+ \log^3 n/\eps^2)$.
\item
For $\log n \leq k \leq \frac{\log n}{\eps^2}$,
$$
g(n,k,\eps)=\Theta (k \log (2+\frac{\log n}{\eps^2 k})).
$$
and the sketch for a given point can be computed in time
linear in its length.
\item
For  $1 \leq k \leq \log n$,
$$
g(n,k,\eps)=\Theta(k \log (1/\eps))
$$
and the sketch for a given point can be computed in time
linear in its length.
\end{itemize}
In all cases the length of the sketch is optimal up to a constant
factor.
\end{theo}

As before, it is convenient to deal with the
different possible ranges for $k$ separately. Note first that
the proof given in Section 2 for the range  $k \leq \log n$ is
essentially constructive, since it is well known (see, for example
\cite{ALSV} or the argument below) that there are explicit constructions of
$\eps$-nets of size $(1/\eps)^{O(k)}$ in $B^k$, and it is enough
to round each vector to a point of the net which is $\eps$-close to
it (and not necessarily to its nearest neighbor).

For  completeness we include a short description of a
$\delta$-net which will also be used later.
For $0 < \delta <1/4$ and for $k \geq 1$ let $N=N(k,\delta)$ denote
the set of all vectors of Euclidean norm at most $1$ in which every
coordinate is an integral multiple of $\frac{\delta}{\sqrt k}$.
Note that each member of $N$ can be represented by $k$ signs
and $k$ non-negative integers $n_i$ whose sum of squares is at most $k/\delta^2$.
Representing each number by its binary representation
(or by two bits, say, if it is $0$ or  $1$)
requires at most
$2 k+\sum_i \log_2 n_i$ bits, where the summation is over all
$n_i \geq 2$. Note that $\sum_i \log_2 n_i =0.5 \log_2 (\Pi_i n_i^2)$
which is maximized when all numbers are equal and gives
an upper bound of $k \log_2 (1/\delta)+2k$ bits per member of the
net. Given a vector in $B^k$ we can round it to a vector
of the net that lies within distance $\delta/2$ from it by simply
rounding each coordinate to the closest integral multiple of
$\delta/\sqrt k$. The computation of the distance between two
points of the net takes time $O(k)$. The size of the net is
$(1/\delta)^k2^{O(k)}$, as each point is represented by
$k \log_2 (1/\delta)+2k$ bits and $k$ signs.

The above description of the net suffices to prove Theorem
\ref{t31} for $k \leq \log n$. We proceed with the proof for
larger $k$.

For $k \geq \frac{40 \log n}{\eps^2}$ we first
apply the Johnson-Lindenstrauss Lemma (with the fast version
described in \cite{AC}) to project the points to $R^m$
for $m=40 \log n/\eps^2$ without changing any square distance or norm
by more than $\eps$. It is convenient to now shrink all vectors
by a factor of $1-\eps$ ensuring they all lie in the unit ball
$B^m$ while the square distances, norms and inner products are still
within $3\eps$ of their original values. We thus may assume from
now on that all vectors lie in $B^m$.

As done in Section 2, we handle norms separately, namely, the
sketch of each vector contains some $O(\log (1/\eps))$ bits
representing a good approximation for its norms. The rest of the
sketch, which is its main part, will be used for recovering
approximate inner products between vectors. This is done
by replacing  each of our vectors $w_i$
by a randomized rounding of it chosen as
follows. Each coordinate of the vector, randomly and independently,
is rounded to one of the two closest integral multiples of
$1/\sqrt m$, where the probabilities are chosen so that its
expectation is the original value of the coordinate.  Thus, if
the value of a coordinate is $(i+p)/\sqrt m$ with $0 \leq p \leq 1$
it is rounded to $i/\sqrt m$ with probability $(1-p)$ and to
$(i+1)/\sqrt m$ with probability $p$. Let $V_i$ be the random vector
obtained from $w_i$ in this way.  Then the expectation of
each coordinate of $V_i-w_i$ is zero. For each $j \neq i$
the random variable $\langle V_i-w_i, w_j \rangle$ is a sum
of $m$ independent random
variables where the
expectation of each of them is $0$ and the sum of squares of the
difference between the maximum value of each random variable and
its minimum value is the square of the norm of $w_j$ divided by
$m$. Therefore this sum is at most $1/m$, and by Hoeffding's
Inequality (see \cite{Ho}, Theorem 2) the probability that this inner
product is in absolute value
at least $\eps/2$ is at most $2e^{-\eps^2 m/8}$ which is
smaller than $1/n^5$. Similar reasoning shows that the
probability that $\langle V_i, V_j-w_j \rangle$
is of absolute value at least
$\eps/2$ is smaller than $1/n^5$. As in the proof in Section 2,
it follows that with probability
at least $1-2/n^3$ all inner products of  distinct vectors in
our rounded set lie within  $\eps$ of their original values,
as needed. The claims about the running time  follow from
\cite{AC} and the description above.
This completes the proof of the first part of Theorem
\ref{t31}.

The proof of the second part is essentially identical
(without the projection step using the Johnson-Lindenstrauss
Lemma). The only difference is in the parameters. If
$k=\frac{40 \delta^2 \log n}{\eps^2}$ with $ \eps \leq \delta \leq
1/2$ we round each coordinate randomly to one of the two
closest integral multiples of $\delta/\sqrt k$, ensuring the
expectation will be the original value of the coordinate.
The desired result follows as before, from the Hoeffding
Inequality. This completes the proof of Theorem \ref{t31}.
\hfill $\Box$

\section{The lower bound}
\begin{lemma}
\label{l31}
If
$$
k=\delta^2 \log n/ (200 \eps^2 )
$$
where $2 \eps \leq \delta \leq 1/2$,
then
$f(n,k,\eps/2)  \geq \Omega(kn \log (1/\delta)$
\end{lemma}
\vspace{0.1cm}

\noindent
{\bf Proof:}\,
Fix a maximal set of points $N$ in the unit ball $B^k$ of $R^k$
so that the
Euclidean distance between any two of them is at least
$\delta$. It is easy and well known that the size of $N$ is
$(1/\delta)^{(1+o(1))k}$ (where the $o(1)$-term tends to $0$ as
$\delta$ tends to $0$). For the lower bound we construct
a large number of $\eps$-separated Gram matrices of $n$ vectors in
$B^k$. Each collection of $n$ vectors consists of a fixed set $R$ of
$n/2$ vectors, whose existence is proved below, together with
$n/2$ points of the set $N$. The set $R$ of fixed points will
ensure that  all the corresponding Gram matrices are $\eps$-separated.

We claim that there is a choice
of a set $R$ of $n/2$ points in $B^k$ so that the
inner products  of any two distinct points from $N$ with some
point of $R$ differ by more than $\eps$. Indeed, for any two
fixed
points of $N$, the difference between them has norm at least
$\delta$, hence
the probability that the product of a random point of $B^k$ with this
difference  is bigger than $\eps$ is at least
$e^{-1.5 \eps^2 k/ \delta^2 }$ (with room to spare).
It thus suffices to have
$$
(1-e^{-1.5 \eps^2 k/ \delta^2 })^{n/2} < 1/|N|^2
$$
hence the following will do:
$$
(n/2) e^{-2 \eps^2 k/ \delta^2 } > (2+o(1)) k \log (1/\delta).
$$
Thus it suffices to have
$$
2 \eps^2 k/ \delta^2 < \log (n/5 k \log (1/\delta))
$$
and as the left hand side is equal to $(\log n)/100$ this indeed
holds. Thus a set $R$ with the desired properties exists.

Fix a set $R$ as above. Note that every two distinct choices of
ordered sets of $n/2$ members of $N$ provide $\eps$-separated
Gram matrices. This implies that
$$
f(n,k,\eps/2) \geq \log
|N|^{n/2}
=\Omega(n \log |N|) = \Omega(n k \log (1/\delta)),
$$
completing the proof of the lemma. \hfill $\Box$

By monotonicity and the case $\delta=1/2$ in the above Lemma
the desired lower bound in Theorem \ref{t11} for all
$k \geq \log n$ follows.

It remains to deal with smaller $k$. Here we fix a
set $N$ of size $(1/ 2 \eps)^{(1+o(1))k}$ in $B^k$ so that the distance
between any two points is at least $2 \eps$. As before,
the inner products with all members of a
random  set $R$ of $n/2$
points distinguishes, with high probability,
between any two members of $N$ by more than
$\eps$. Fixing $R$  and adding to it in all possible ways
an ordered set of $n/2$ members  of $N$ we conclude that
in this range
$$
f(n,k,\eps/2) \geq \log (|N|^{n/2})=
\Omega( nk \log (1/\eps))
$$
completing the proof of the lower bound and hence that of Theorem
\ref{t11}.  \hfill $\Box$

We conclude this section by observing that the proof of the lower
bound implies that the size of the sketch per point given by Theorem
\ref{t31} is tight, up to a constant factor, for all admissible
values of the parameters. Indeed, in the lower bounds we  always
have a fixed set $R$ of $n/2$ points and a large net $N$, so that
if our set contains all the points of $R$ then no two distinct
points of $N$ can have the same sketch, as for any two distinct
$u,v \in N$  there is a
member of $R$ whose inner products with $u$ and with $v$ differ by
more than $\eps$. The lower bound for the length of the sketch is
thus $\log |N|$, by the pigeonhole principle.

\section{Small distortion}

In this section we prove several
results related  the case of smaller $\eps$. In Section 6.1 we
prove a tight estimate for the number
of bits needed to represent $\eps$-approximations of all inner products
$\langle a_1, y \rangle,\ldots, \langle a_n, y \rangle$
for a vector $y \in \RR^k$ of norm at most $1$,
where $a_1,a_2, \ldots ,a_n \in \RR^k$
are fixed vectors of norm at most $1$.
In Section 6.2 we present the proof of Theorem \ref{t23}.
In Section 6.3 we prove Theorem \ref{thm_944}, while in Section 6.4
we prove Theorem \ref{t13}. The techniques here are more
sophisticated than those in the previous sections, and rely on several
tools from convex geometry.

\subsection{Inner products with fixed vectors}

\begin{theo}
\label{t42}
Let $a_1,a_2, \ldots ,a_n$ be vectors of norm at most $1$ in
$R^k$. Suppose $\eps  \geq \frac{2}{\sqrt n}$ and assume that
$$
\frac{\log (2+\eps^2 n)}{8\eps^2} \leq k \leq n.
$$
Then, for a vector $y$ of norm at most $1$
the number of bits required to represent all inner products
$\langle a_i, y \rangle$ for all
$1 \leq i \leq n$ up to an additive error of
$\eps$ in each such product is
$$
\Theta \left( \frac{\log (2+\eps^2 n)}{\eps^2} \right).
$$
Equivalently, the number of possibilities of the vector
$$
\left(\lfloor \frac{\langle a_1,y \rangle}{\eps} \rfloor,
\lfloor \frac{\langle a_2,y \rangle}{\eps} \rfloor,  \cdots ,
\lfloor \frac{\langle a_n,y \rangle}{\eps} \rfloor \right)
$$
for vectors $y$ of norm at most $1$ is
$$
2^{\Theta( \frac{\log (2+\eps^2 n)}{\eps^2})}.
$$
\end{theo}
\vspace{0.1cm}

\noindent
{\bf Proof:}\, As the number of bits required is clearly a monotone
increasing function of the dimension it suffices to prove the upper
bound for $k=n$ and the lower bound for
$k=\frac{\log (2+\eps^2 n)}{8\eps^2}$.

We start with the upper bound. Define $t>0$ by the equation
$$
\eps=\frac{\sqrt {2 \log (2+n/t)}}{\sqrt t}.
$$
(There is a unique solution as the right hand side is a decreasing
function of $t$).
Therefore
$$
t=\frac{2 \log (2+n/t)}{\eps^2}.
$$
Since $\eps \geq \frac{2}{\sqrt n}$ this implies that
$t <n$ since otherwise the right hand side is at most
$2 \log 3 \cdot n/4 <n$.
By the last expression for $t$, $t \geq \frac{1}{\eps^2}$ and thus
$ \log (2+n/t) \leq \log (2+\eps^2 n)$ implying that
$$
t \leq \frac{2 \log (2+\eps^2 n)}{\eps^2}.
$$
This implies that
$$\frac{n}{t} \geq \frac{\eps^2 n}{2 \log (2+\eps^2 n)} $$
and since $\eps^2 n \geq 4$ it follows that
$$
\log (2+n/t) \geq \frac{1}{4} \log (2 +\eps^2 n),
$$
as can be shown by checking that for $z \geq 4$,
$$
2+\frac{z}{2 \log (2+z)} \geq (2+z)^{1/4}.
$$
We have thus shown that
$$
\frac{\log (2 +\eps^2 n)}{2\eps^2} \leq t \leq
\frac{2 \log (2 +\eps^2 n)}{\eps^2}.
$$

Define a convex set $K$ in $R^n$ as follows.
$$
K=\{ x \in R^n: ~|\langle \frac{x}{\sqrt t}, a_i \rangle | \leq
\eps ~~\mbox{for all}~~1 \leq i \leq n \}.
$$
By the Khatri-Sidak Lemma (\cite{Kh}, \cite{Si}, see also \cite{G}
for a simple proof), if $\gamma_n $
denotes the standard Gaussian measure in $R^n$, then
$$
\gamma_n(K) \geq \prod_{i=1}^n
\gamma_n(\{ x \in R^n: ~|\langle \frac{x}{\sqrt t}, a_i \rangle |
\leq \eps\} ) \geq (1-2e^{-\eps^2 t/2})^n
$$
$$
\geq(1-2e^{-\log(2+n/t)})^n =(1-\frac{2t}{2t+n})^n \geq e^{-3t}.
$$
For every measurable centrally symmetric set $A$ in $R^n$ and for
any vector $x \in R^n$,
$$
\gamma_n(x+A) \geq e^{-\|x\|^2/2} \gamma_n(A).
$$
For completeness we repeat the standard argument.
$$
\gamma_n(x+A) =\int_A e^{-\|x+y\|^2/2} \frac{1}{(2\pi)^{n/2}} dy
=e^{-\|x\|^2/2}\gamma_n(A) \int_A e^{-\langle x,y\rangle}
e^{-\|y\|^2/2} \frac{1}{\gamma_n(A) (2\pi)^{n/2}} dy.
$$
The integral in the right hand side is the expectation, with
respect to the Gaussian measure on $A$, of
$e^{-\langle x,y\rangle}$. By Jensen's Inequality this is at least
$e^z$ where $z$ is the expectation  of
$-\langle x,y\rangle$ over $A$. As $A=-A$ this last expectation
is $0$ and as $e^0=1$ we conclude that
$\gamma_n(x+A) \geq e^{-\|x\|^2/2}\gamma_n(A)$, as needed.
Taking $A$ as the set $K$ defined above and letting $x$ be any
vector $b$ of norm at most $1$ in $R^n$ we get
$$
\gamma_n(\sqrt t b +K) \geq e^{-t/2}\gamma_n(K) > e^{-4t}.
$$

Given a vector $b \in R^n$, $\|b\| \leq 1$, let $X$ be a standard random
Gaussian in $R^n$. We bound from below the probability of the event
$E_b$
that for every $i$, $1 \leq i \leq n$,
$$
| \langle \frac{X}{\sqrt t},a_i \rangle -\langle b,a_i \rangle |
\leq \eps.
$$
This, however, is exactly the probability that
$X-b \sqrt t \in K$, that is, $\gamma_n(\sqrt t b+K)$
which as we have
seen is at least $e^{-4t}$.

We can now complete the proof of the upper bound as done in Section
2. Let $B$ be a maximum collection of vectors of norm at most $1$ in $R^n$
so that for every two distinct $b,b' \in B$ there is some $i$ so
that $|\langle b,a_i \rangle -\langle b',a_i \rangle| >2 \eps$.
Then the events $E_b$ for $b \in B$ are pairwise disjoint  and
hence the sum of their probabilities is at most $1$. It follows
that $|B| \leq e^{4t}$. The upper bound follows as the number of
bits needed to represent all inner products
$\langle b,a_i \rangle$ for $1 \leq i \leq n$ up to an additive
error of $2 \eps$ is at most $\lceil \log_2 |B| \rceil$.

We proceed with the proof of the lower bound, following the
reasoning in Section 4. Put
$$
k= \frac{\log (2+\eps^2 n)}{8\eps^2}.
$$
Let $B$ be a collection of, say, $e^{k/8}$ unit vectors in $R^k$ so that
the Euclidean distance between any two of them is at least $1/2$.
We claim that there are $n$ unit vectors $a_i$ in $R^k$ so that for
any two distinct members $b,b'$ of $B$ there is an $i$ so that
$|\langle b,a_i \rangle -\langle b',a_i \rangle| >\eps$.

Indeed, taking the vectors $a_i$ randomly, independently
and uniformly in the unit ball of $R^k$ the probability that for a
fixed pair $b,b'$ the above fails is at most
$$
(1-e^{-4\eps^2 k})^n.
$$
Our choice of parameters ensures that
$$
{|B| \choose 2} (1-e^{-4\eps^2 k})^n <1.
$$
Indeed it suffices to check that
$$
e^{-4\eps^2 k} \cdot n > k/4
$$
that is
$ 4\eps^2 k<\log (4n/k)$ or
$$
k < \frac{\log
(4n/k)}{4\eps^2}.
$$
It thus   suffices to check that
$$
\log (2+\eps^2 n) < 2 \log (4n/k)=2 \log (\frac{32\eps^2
n}{\log(2+\eps^2n)}).
$$
This easily holds since for $\eps \geq 2/\sqrt n$,
$$
2 \log (\frac{32\eps^2 n}{\log(2+\eps^2n)})
> \log (2+\eps^2 n).
$$

By the union bound the assertion of the claim follows, implying the
desired lower bound  as  no two members of $B$ can have the same
representation. This completes the proof of the theorem. \hfill
$\Box$

\subsection{Compression schemes}

In this subsection we prove Theorem \ref{t23}. The basic approach
is similar to the one in the proof of Theorem \ref{t11}, the main
difference is that in the upper bound proved in Lemma \ref{l21}
we replace the simple union bound by a more sophisticated geometric
argument based on Harg\'e's Inequality, which is a special case of the
Gaussian correlation conjecture, proved recently by Royen.
We start with the following Lemma.
\begin{lemma}
\label{l62}
Let $H_1,\ldots,H_n \subseteq \RR^k$ be symmetric slabs, where
a symmetric slab is a set of the form $\{ x \in \RR^k \, ; \, |\langle x, \theta \rangle| \leq 1 \}$ for some $\theta \in \RR^k$.
Then,
$$ \frac{Vol_k \left( B^k \cap \bigcap_{i=1}^n H_i \right)}{Vol_k(B^k)}
\geq c^k \prod_{i=1}^n \gamma_k(\sqrt{k} H_i ), $$
where $c>0$ is an absolute constant.
\end{lemma}
\begin{proof}
Since $T = \sqrt{k} \bigcap_{i=1}^n H_i $ is convex and
centrally-symmetric, we may  use Harg\'e's inequality \cite{Ha}, which
is a particular case of the Gaussian correlation
inequality proven by Royen \cite{Ro}. This implies that
$$ \gamma_k \left( \sqrt{k} \left( B^k \cap \bigcap_{i=1}^n
H_i \right) \right) \geq \gamma_k(\sqrt{k} B^k) \cdot
\gamma_k \left( \sqrt{k} \bigcap_{i=1}^n H_i \right)
\geq c \prod_{i=1}^n \gamma_k(\sqrt{k} H_i ) $$
where the last passage is the Khatri-Sidak lemma. However,
$$ \frac{Vol_k \left( B^k \cap \bigcap_{i=1}^n H_i \right)}
{Vol_k(B^k)} = \frac{Vol_k \left( \sqrt{k} \left( B^k
\cap \bigcap_{i=1}^n H_i \right) \right)}{Vol_k(\sqrt{k} B^k)}
\geq  \frac{\gamma_k \left( \sqrt{k} \left( B^k \cap
\bigcap_{i=1}^n H_i \right)\right) (2 \pi)^{k/2}}
{Vol_k(\sqrt{k} B^k)}
$$
since the density of $\gamma_k$ is at
most $(2\pi)^{-k/2}$.
Since $Vol_k(\sqrt{k} B^k) \leq C^k$, the lemma is proven.
\end{proof}

We proceed with the proof of the upper bound in Theorem \ref{t23}.
For $t \leq k \leq n$ the upper bound (which is probably not tight)
is proved by repeating the proof of Lemma \ref{l21} as it is.
For $1 \leq k \leq \log(2+\eps^2 n)$ the upper bound follows
by rounding each vector to the closest point in an $\eps$-net in
the ball $B^k$. It remains to deal with the interesting range
$\log(2+\eps^2 n) \leq k \leq t$. By the computation in the
beginning of the proof
of Theorem \ref{t42},
$$
\eps=\Theta (\frac{\sqrt { 2 \log (2+n/t)}}{\sqrt t}).
$$
Suppose $k=\delta t$, with $\eps^2 \leq \delta \leq b$ for some
small absolute positive constant $b$. Given  points
$w_1,\ldots,w_n$ in $B^k$, as in
the proof of Lemma \ref{l21} it suffices
to prepare a
sketch for the inner products between pairs of distinct points.
Again, as  in that proof, let
$\GG$ be a maximal
(with respect to containment) set of Gram matrices  of ordered
sequences of $n$ vectors $w_1, \ldots ,w_n$ in $B^k$,
so that every
two distinct members of $\GG$ are $\eps$-separated (that is, have
at least one non-diagonal entry in which the two matrices differ by
more than $\eps$). By
the maximality of $\GG$, for every Gram matrix $M$
of $n$ vectors in $B^k$
there is a member of $\GG$ in which
all inner products of pairs of distinct points
are within $\eps$ of the corresponding inner products
in $M$, meaning that as a sketch for $M$ it suffices to
store (besides the approximate norms of the vectors),
the index of an appropriate member of $\GG$. This requires
$\log |\GG|$ bits. It remains to prove an upper bound for
the cardinality of $\GG$.

Let $V_1,V_2, \ldots ,V_n$ be $n$ vectors, each chosen
randomly, independently and uniformly in the ball of
radius $2$ in $R^k$. Let $T=G(V_1,V_2, \ldots ,V_n)$
be the Gram matrix of the vectors $V_i$.
For each $G \in \GG$ let $A_G$ denote the
event that for every $1 \leq i \neq j \leq n$,
$|T(i,j)-G(i,j)| < \eps/2$. Note that since the members of
$\GG$ are $\eps$-separated, all the events $A_G$ for $G \in \GG$
are pairwise disjoint. To complete the proof it thus suffices to
show that the probability of each event
$A_G$ is at least $e^{-O(nk \log (1/\delta))}$. To see that
this is the case,
fix a Gram matrix $G = G(w_1,\ldots,w_n) \in \GG$ for
some $w_1,\ldots,w_n \in B^k$
of norm at most $2$. For each fixed
$i$ the probability that
$V_i$ lies in the ball of radius $\delta$
centered at $w_i$ is exactly
$(\delta/2)^k$. Therefore the probability that this happens for all $i$
is $(\delta/2)^{nk}$.
Conditioning on that, for each $i$ the
vector $V_i-w_i$ is uniformly  distributed
in the ball of radius $\delta$ in $R^k$
centered at $0$. For each $i$ let, now,
$A_i$ be the event that $|\langle V_i-w_i, w_j \rangle| \leq
\eps/4$ for all $i < j \leq n$, and that $|\langle V_{\ell},
V_i-w_i \rangle | \leq \eps/4$ for all $1 \leq \ell <i$.
In particular, the event $A_1$ is that $V_1-w_1$ lies in the intersection
of the $n-1$ slabs $|\langle x,w_j \rangle| \leq \eps/4$ for $j>1$. More
generally, conditioning on the events
$A_1,\ldots A_{i-1}$ (as well as on the events that
$|V_j-w_j| \leq \delta$ for all $j$), the event $A_{i}$ is that
$V_{i}-w_{i}$ lies in the intersection of the slabs
$|\langle x,w_j \rangle| \leq \eps/4$ for $j > i$ and the slabs
$|\langle V_{\ell},x \rangle | \leq \eps/4$ for $\ell <i$. Note
that conditioning on $A_1,\ldots ,A_{i-1}$, the
vectors $V_1,V_2, \ldots ,V_{i-1}$ are of norm at most
$1+\eps/4<2$, and once their values are exposed then indeed
we have here an intersection of $n-1$ slabs with a ball centered at
the origin.

By Lemma \ref{l62} it follows that the conditional probability of
each event $A_i$ given all previous ones $A_1,\ldots ,A_{i-1}$ and
given that all vectors $V_i$ lie within distance $\delta$ of
the corresponding vectors $w_i$ is at least
$$
C^{-k} (1-2 e^{-\frac{\eps^2}{64\delta^2} k})^n,
$$
where $C$ is an absolute positive constant.
Since $k=\delta t$ and $\eps^2 t=\Theta(\log (2+n/t))$
it follows that
$$
e^{-\frac{\eps^2}{64\delta^2} k} \leq
e^{-\frac{c\log(2+n/t)}{\delta}}=
(\frac{t}{2t+n} )^{c/\delta}.
$$
As $t \leq n$ the last quantity is at most
$$
(\frac{1}{3})^{\frac{c}{2\delta}} \frac{t}{2t+n}
$$
provided $\delta <c/2$.

Thus
$$
(1-2 e^{-\frac{\eps^2}{64\delta^2} k})^n
\geq
[1-(\frac{1}{3})^{c/2\delta} \frac{t}{2t+n}]^n \geq
e^{-(1/3)^{c/2\delta} t} \geq e^{-\delta t} =e^{-k}
$$
for all $\delta <c'$.

By multiplying all conditional probabilities we conclude that
the probability that $V_i-w_i$ is of norm at most $\delta$ for all
$i$ and that all events $A_i$ hold too is at least
$e^{-O(n k \log (1/\delta))}$. However, in this case, for all $i <j$
$$
|\langle V_i,V_j \rangle -\langle w_i,w_j \rangle |
\leq
|\langle V_i-w_i,w_j \rangle|+
|\langle V_i,V_j-w_j \rangle |
\leq \eps/2
$$
and the event $A_G$ occurs. Thus the probability of
each event $A_G$ is at least
$e^{-O(n k \log (1/\delta))}$, providing the required upper bound for
$|\GG|$ and hence completing the proof of the upper bound in
Theorem \ref{t23}.

The proof of the lower bound is similar to the proof of the lower
bound in Theorem \ref{t42}. The most interesting case here
is again the range
$$
\log (2+\eps^2 n) \leq k \leq t=\frac{\log (2+\eps^2
n)}{\eps^2}.
$$
(Note that the lower bound for $k \geq t$ follows from the case
$k=\Theta(t)$.)
Here it is convenient to define
$\delta$ so that $k=\delta^2 \frac{\log (2+\eps^2 n)}{4\eps^2}$
where $\eps \geq \frac{2}{\sqrt n}$ and $2\eps \leq \delta <1$
and to assume we have $2n$ points.
Let $B$ be a collection of, say,
$(\delta^{-1} / 2)^k$ unit vectors in $R^k$ so that
the Euclidean distance between any two of them is at least $\delta$.
We claim that there are $n$ unit vectors $a_i$ in $R^k$ so that for
any two distinct members $b,b'$ of $B$ there is an $i$ so that
$|\langle b,a_i \rangle -\langle b',a_i \rangle| >\eps$.

Indeed, taking the vectors $a_i$ randomly, independently
and uniformly in the unit ball of $R^k$ the probability that for a
fixed pair $b,b'$ the above fails is at most
$$
(1-e^{-\frac{\eps^2}{\delta^2} k})^n.
$$
Our choice of parameters ensures that
$$
{|B| \choose 2} (1-e^{-\frac{\eps^2}{\delta^2} k})^n <1.
$$
Indeed it suffices to check that
$$
e^{-\frac{\eps^2}{\delta^2}k} \cdot n > 2 k \log (1/2\delta)
$$
that is
$$
\frac{\eps^2}{\delta^2}k < \log (\frac{n}{2k \log (1/2\delta)}),
$$
or equivalently
$$
k < \frac{\delta^2}{\eps^2} \log (\frac{n}{2k \log (1/2\delta)}).
$$
By the definition of
$$
k=\delta^2 \frac{\log (2+\eps^2 n)}{4 \eps^2}
$$
it suffices to show that
$$
\log (2+\eps^2 n) < 4 \log (\frac{n}{2k \log (1/2\delta)})
=4 \log [\frac{n 4 \eps^2}{2\delta^2 \log (2+\eps^2 n) \log
(1/2\delta)}].
$$
This easily holds for $\eps \geq 2/\sqrt n$.

By the union bound the assertion of the claim follows.
The desired result now holds, since every union of the vectors
$a_i$ with an
ordered set of $n$ members of $B$ must have a different
representation, hence the number of bits needed is at least
$n \log_2 |B|=\Omega(nk \log (1/\delta)$.

The case $k \leq \log (2+\eps^2 n)$ is proved in a similar way
by letting $B$ be a $2 \eps$-separated set of points in $B^k$.
We omit the detailed computation.
This completes the proof of the theorem. \hfill $\Box$

\subsection{Halving the dimension}

In this subsection we prove Theorem \ref{thm_944}.
The theorem
is equivalent to the following statement:

\begin{theo}
Let $m \geq n \geq 1, \eps > 0$ and assume that
$a_1,\ldots,a_m, b_1,\ldots,b_m \in \RR^{2n}$
are points of norm at most one.
Suppose that $X_1,\ldots,X_m, Y_1,\ldots,Y_m \in \RR^n$
are independent random vectors, distributed according to
standard Gaussian law.

\medskip Assume that $n \geq C_1 \cdot \eps^{-2} \log(2 + \eps^2 m)$.
Then with probability of at least $\exp(-C_2 n m)$,
\begin{equation}
 \left| \left \langle \frac{X_i}{\sqrt{n}}, \frac{Y_j}{\sqrt{n}}
\right \rangle - \left \langle a_i, b_j \right \rangle \right |
\leq \eps   \qquad \qquad \text{for} \ i,j=1,\ldots,m, \label{eq_1126}
\end{equation}
and moreover $|X_i| + |Y_i| \leq C_3 \sqrt{n}$ for all $i$.
\label{prop_948_}
\end{theo}

In the proof of Theorem \ref{prop_948_} we will use the following
theorem, which  is the dual version of the finite-volume ratio theorem
of Szarek and Tomczak-Jaegermann (see e.g. \cite[Section 5.5]{AGM}
and also \cite{K} for an alternative proof).
A convex body is a compact, convex set with a
non-empty interior, and as before
$B^n = \{ x \in \RR^n \, ; \, |x| \leq 1 \}$
is the centered unit Euclidean ball in $R^n$.

\begin{theo} Let $K \subseteq B^{2n}$ be a centrally-symmetric
convex body with $Vol_{2n}(K) \geq e^{-10 n} Vol_{2n}(B^{2n})$.
Then there exists an $n$-dimensional subspace $E \subseteq \RR^{2n}$
with
$$ c B^{2n} \cap E \subseteq Proj_E(K),
$$
where $Proj_E$ is the orthogonal projection
operator onto $E$ in $\RR^{2n}$.
\label{thm_fvr}
\end{theo}

For completeness we include a short derivation of this theorem
from \cite[Theorem 5.5.3]{AGM}.
\begin{proof}

The polar body to a centrally-symmetric convex
body $K \subseteq R^{2n}$ is $$ K^{\circ}
= \{ x \in \RR^{2n} \, ; \, \forall y \in K, \,
|\langle x, y \rangle| \leq 1  \}. $$
Polarity is an order-reversing involution, i.e., $(K^{\circ} )^{\circ}
= K$ while $K_1 \subseteq K_2$ implies that
$K_1^{\circ} \supseteq K_2^{\circ}$.
Moreover, $(B^{2n})^{\circ} = B^{2n}$.
Since $K \subseteq B^{2n}$
we know that $B^{2n} \subseteq K^{\circ}$.
By the Santal\'o inequality (e.g., \cite[Theorem 1.5.10]{AGM}),
$$  Vol_n(K^{\circ}) \leq \frac{Vol_{2n}(B^{2n})^2}
{Vol_n(K)} \leq e^{10 n} Vol_{2n}(B^{2n}). $$
According to the finite-volume ratio theorem (see,
e.g., \cite[Theorem 5.5.3]{AGM}),
there exists an $n$-dimensional subspace $E \subseteq \RR^{2n}$ with
\begin{equation} K^{\circ} \cap E \subseteq C (B^{2n} \cap E).
\label{eq_1221} \end{equation}
However, $Proj_E(K)^{\circ} = K^{\circ} \cap E$
for any subspace $E \subseteq \RR^{2n}$.
Thus the desired conclusion  follows from (\ref{eq_1221}).
\end{proof}

Theorem \ref{thm_fvr} implies the following:

\begin{lemma}
Let $K$ be as in Theorem \ref{thm_fvr}. Then there exists
an $n$-dimensional subspace $E \subseteq \RR^{2n}$ so that
$$  \forall x \in c_1 B^{2n}, \quad Vol_n( E \cap (x + K) )
\geq c_2^n \cdot Vol_n(B^n).
$$
\label{lem_1006}
\end{lemma}

\begin{proof} We set $c_1 = c/2$ where $c > 0$ is the constant from
Theorem \ref{thm_fvr}. Thus there exists an $n$-dimensional
subspace $E$ with
\begin{equation}  2 c_1 B^{2n} \cap E^{\perp}
\subseteq Proj_{E^{\perp}}(K) \label{eq_1056}
\end{equation}
where $E^{\perp}$ is the orthogonal
complement to $E$ in $\RR^{2n}$. By Fubini's theorem,
\begin{equation} e^{-10 n} \cdot Vol_{2n}(B^{2n}) \leq
Vol_{2n}(K) \leq Vol_n(Proj_{E^{\perp}}(K))
\cdot \sup_{x \in E^{\perp}} Vol_n(E \cap (x + K)). \label{eq_850}
\end{equation}
The Brunn-Minkowski inequality and the
central symmetry of $K$ imply that for any $x \in E^{\perp}$,
$$ Vol_n(E \cap K)^{\frac{1}{n}} \geq  \frac{Vol_n(E
\cap (x + K))^{\frac{1}{n}} + Vol_n(E \cap (-x + K))^{\frac{1}{n}}}{2}
= Vol_n(E \cap (x + K))^{\frac{1}{n}}.
$$
Thus the supremum in (\ref{eq_850}) is attained for $x = 0$.
Since $K \subseteq B^{2n}$ we conclude from (\ref{eq_850}) that
\begin{equation}
Vol_n(K \cap E) \geq e^{-10 n} \cdot
\frac{Vol_{2n}(B^{2n})}{Vol_n(B^n)} \geq e^{-C n}
\cdot Vol_n(B^n),
\label{eq_851}
\end{equation}
for some constant $C > 0$. Let $x \in \RR^{2n}$ satisfy $|x| \leq c_1$.
Then $Proj_{E^{\perp}}(-2x) \in 2 c_1 B^{2n} \cap E^{\perp}$.
According to (\ref{eq_1056}) there exists $y \in K$ with $y + 2x \in E$.
Thus $(\{ y \} + K \cap E) / 2 \subseteq K \cap (E - x)$.
By (\ref{eq_851}) and the convexity of $K$,
$$ Vol_n( E \cap (x + K) ) = Vol_n(K \cap (E - x)) \geq Vol_n
\left( \frac{\{ y \} + K \cap E}{2} \right) \geq c_2^n
\cdot Vol_n(B^n),
$$
for some constant $c_2 > 0$, completing the proof of the lemma.
\end{proof}

As in the previous subsection
we write $\gamma_{2n}$ for the standard Gaussian probability
measure in $\RR^{2n}$.
For a subspace $E \subseteq \RR^{2n}$, write $\gamma_E$
for the standard Gaussian measure in the subspace $E$. For
$K \subset \RR^{2n}$ we denote $\gamma_E(K \cap E)$ by
$\gamma_E(K)$.

\begin{coro}
Let $K \subseteq \RR^{2n}$ be a centrally-symmetric convex
body with $\gamma_{2n}(K) \geq e^{-n}$.
Then there exists an $n$-dimensional subspace
$E \subseteq \RR^{2n}$ such that for any $v \in \RR^{2n}$,
$$ |v| \leq \sqrt{n} \qquad \Longrightarrow \qquad
\gamma_E( v + C K ) \geq c^n. $$
\label{cor_1024}
\end{coro}

\begin{proof}
Write $\sigma_{2n-1}$ for the uniform probability measure on the
unit sphere $S^{2n-1} = \{ x \in \RR^{2n} \, ; \, |x| = 1 \}$.
For $K \subset \RR^{2n}$ denote $\sigma_{2n-1}(K \cap S^{2n-1})$
by $\sigma_{2n-1}(K)$.
Since $K$
is a convex set containing the origin and the Gaussian measure is
rotationally-invariant, for any $r > 0$,
$$
e^{-n} \leq \gamma_{2n}(K) \leq \gamma_{2n}(r B^{2n})
+ \gamma_{2n}(K \setminus r B^{2n}) \leq \gamma_{2n}(r B^{2n})
+ \sigma_{2n-1} \left( \frac{K}{r} \right).
$$
A standard estimate shows that
$\gamma_{2n} ( c_1 \sqrt{n} B^{2n}) \leq e^{-n}/2$
for some universal constant $c_1 > 0$.
It follows that for $K_1 = K \cap c_1 \sqrt{n} B^{2n}$,
$$
\frac{Vol_{2n}(K_1)}{Vol_{2n}(c_1 \sqrt{n} B^{2n})}
\geq \sigma_{2n-1} \left( \frac{K_1}{c_1 \sqrt{n}} \right)
= \sigma_{2n-1} \left( \frac{K}{c_1 \sqrt{n}} \right)
\geq e^{-n} / 2.
$$
By Lemma \ref{lem_1006}, there exists an $n$-dimensional subspace $E \subseteq \RR^{2n}$
such that
$$
\forall x \in c_2 B^{2n}, \quad Vol_n \left( E \cap \left(x
+ \frac{K_1}{c_1 \sqrt{n}} \right) \right)
\geq c^n \cdot Vol_n(B^n) \geq \left( \frac{\tilde{c}}{\sqrt{n}} \right)^n.
$$
Now that the universal constants $c_1$ and $c_2$ are
determined, we proceed as follows: For any $v \in \RR^{2n}$
with $|v| \leq \sqrt{n}$,
\begin{align*}
\gamma_E \left( v + \frac{K}{c_1 c_2} \right) &
\geq \gamma_E \left( v + \frac{K_1}{c_1 c_2} \right)
\geq e^{-C n} Vol_n  \left( E \cap \left( v + \frac{K_1}{c_1 c_2}
\right) \right) \\
& \geq \left( \tilde{c} \sqrt{n} \right)^n Vol_n
\left( E \cap \left( \frac{c_2 v}{ \sqrt{n}}
+ \frac{K_1}{c_1 \sqrt{n}} \right) \right)
\tag*{\qedhere} \geq \bar{c}^n.
\end{align*}
\end{proof}

\bigskip
Before continuing with the proof of Theorem \ref{prop_948_}
recall that as mentioned in the previous subsection,
for any $a \in \RR^n$ and a centrally-symmetric
measurable set $T \subseteq \RR^n$,
$$ \gamma_n(T + a) \geq e^{-\|a\|^2/2} \gamma_n(T). $$

\bigskip \begin{proof}[Proof of Theorem \ref{prop_948_}]
We may assume that $n \geq 5 \cdot \eps^{-2} \log(2 + \eps^2 m)$, thus
$$
\eps \geq  2 \sqrt{\frac{\log (2 +  m/n)}{n}}.
$$
We identify $\RR^n$ with the subspace of $\RR^{2n}$ of all vectors
whose last $n$ coordinates vanish,
thus we may write $\RR^n \subseteq \RR^{2n}$.
Let $U \in O(2n)$ be an orthogonal matrix to
be determined later on. Observe that for all $i,j$,
\begin{equation} \left| \left \langle \frac{X_i}{\sqrt{n}}, \frac{Y_j}
{\sqrt{n}} \right \rangle - \left \langle a_i, b_j
\right \rangle \right | \leq
\left| \left \langle \frac{U X_i}{{\sqrt{n}}} - a_i, b_j
\right \rangle \right | + \left| \left \langle \frac{X_i}{\sqrt{n}},
\frac{Y_j}{\sqrt{n}} - U^{-1} b_j \right \rangle \right |.
\label{eq_936}
\end{equation}
We shall bound separately each of the two summands on the
right-hand side of (\ref{eq_936}). Define
$$ K = \left \{ x \in \RR^{2n} \, ; \, \left| \left \langle
\frac{x}{\sqrt{n}}, b_j \right \rangle \right|
\leq \eps \, \ \text{for} \ j=1,\ldots,m \right \}. $$
Recall that $\Phi(t) = (2 \pi)^{-1/2} \int_t^{\infty}
\exp(-s^2/2) ds$
and $\Phi(t) \leq \exp(-t^2/2)$ for $t \geq 1$.
By the Khatri-Sidak lemma
\begin{align*} \gamma_{2n}(K) & \geq \prod_{j=1}^m
\gamma_{2n} \left( \left \{ x \in \RR^{2n} \, ; \, \left|
\left \langle \frac{x}{\sqrt{n}}, b_j \right \rangle \right|
\leq \eps \right \} \right)
= \prod_{j=1}^{m} \left( 1 - 2 \Phi(\sqrt{n} \eps / |b_j|) \right)
\\ &
 \geq \left(1 - 2 \Phi \left( 2 \sqrt{\log (2+m/n)} \right) \right)^{m}
\geq \left(1 - \frac{n}{m+n} \right)^m \geq e^{-n}. \nonumber
\end{align*}
From Corollary \ref{cor_1024}, there exists an $n$-dimensional
subspace $E \subseteq \RR^{2n}$ such that for any $v \in \RR^{2n}$
\begin{equation}  |v| \leq \sqrt{n} \qquad \Longrightarrow \qquad
\gamma_E( v + C K ) \geq c^n. \label{eq_1116} \end{equation}
Let us now set $U \in O(2n)$ to be any orthogonal
transformation with $U(\RR^n) = E$.
We  also set $C_3$ to be a sufficiently large universal
constant such that $\PP(|X_i| \leq C_3 \sqrt{n} ) \geq 1 - c^n/2$,
where $c > 0$ is the constant from (\ref{eq_1116}).
Then
\begin{align} \label{eq_1107} \PP & \left(\forall i, \ U X_i -
\sqrt{n} a_i \in C K \quad \text{and}
\quad |X_i| \leq C_3 \sqrt{n} \right)  \\ &
= \prod_{i=1}^m \gamma_E \left(
\left( \sqrt{n} a_i + C K  \right) \cap C_3 \sqrt{n} B^{2n} \right)
 \geq \exp(-\hat{C} n m). \nonumber \end{align}
 We move on to bounding the second summand on the
right-hand side of (\ref{eq_936}). We condition
on the $X_i$'s satisfying the event described in (\ref{eq_1107}).
In particular, $|X_i| \leq C_3 \sqrt{n}$ for all $i$.
We now define
$$
T = \left \{ y \in \RR^{n} \, ; \, \left| \left \langle
\frac{y}{\sqrt{n}}, \frac{X_i}{\sqrt{n}} \right \rangle \right|
\leq \eps \, \ \text{for} \ i=1,\ldots,m \right \}.
$$
Arguing as before, we deduce from the Khatri-Sidak
lemma that $\gamma_{n}(T) \geq e^{- C n}$.
Write $P(x_1,\ldots,x_{2n}) = (x_1,\ldots,x_n)$. Then for any $j$,
$$
\PP \left( \forall i, \left| \left \langle
\frac{X_i}{\sqrt{n}}, \frac{Y_j}{\sqrt{n}} - U^{-1} b_j
\right \rangle \right| \leq \eps \right) =
\gamma_n \left( T + \sqrt{n} P(U^{-1} b_j) \right)
\geq e^{-n \|b_j\|^2/2} \gamma_n(T) \geq e^{-\tilde{C} n}.
$$
Next we set $\tilde{C}_3$ to be a sufficiently large universal
constant such that $\PP(|Y_i| \leq \tilde{C}_3 \sqrt{n} )
\geq 1-\exp(-\tilde{C} n)/2$.

\medskip To summarize,  with probability at
least $\exp(- \hat{C} n m)$, for all $i,j$,
$$ U X_i - \sqrt{n} a_i \in C K, \quad \left| \left \langle
\frac{X_i}{\sqrt{n}}, \frac{Y_j}{\sqrt{n}} - U^{-1} b_j
\right \rangle \right| \leq \eps  \quad \text{and} \quad |X_i| + |Y_i| \leq \tilde{C} \sqrt{n}. $$
We thus have an upper bound of $\bar{C} \eps$
for the right-hand side of (\ref{eq_936}) for all $i,j$,
and moreover, $|X_i| + |Y_i| \leq \tilde{C} \sqrt{n}$ for all $i$.
This implies a variant of Theorem \ref{prop_948_},
in which the $\eps$ in (\ref{eq_1126}) is replaced by $\bar{C} \eps$.
However, by adjusting the constants, this variant is
clearly  seen to be
equivalent to the original formulation, and the proof is complete.
\end{proof}

\subsection{Keeping the inner products with small distortion}

In this subsection we prove Theorem \ref{t13}.
The
main result we use is the well-known low $M^*$-estimate
due to Pajor and Tomczack-Jaegermann, which builded upon earlier
contributions by Milman and by Gluskin, see e.g.,
\cite[Chapter 7]{AGM}:

\begin{theo}
Let $1 \leq t \leq n$ and let $K \subseteq \RR^n$ be a
centrally-symmetric convex body with $\gamma_n(K) \geq 1/2$.
Let $E \subseteq \RR^n$ be a random subspace of dimension $n - t$.
Then with probability at least $1 - C \exp(-c t)$ of selecting $E$,
$$ \tilde{c} \sqrt{t} B_E \subseteq Proj_E( K ). $$
Here, $c, \tilde{c}, C  > 0$ are universal constants
and $B_E = B^n \cap E$.
\label{thm_1405}
\end{theo}

\begin{proof} Our formulation is very close to (7.1.1)
and Theorem 7.3.1 in \cite{AGM}. We only need to explain a standard fact,
why  $\gamma_n(K) \geq 1/2$
implies the bound $M(K) \leq C / \sqrt{n}$ where
$$ M(K) := \int_{S^{n-1}} \| x \|_K d \sigma_{n-1}(x) $$
and $\| x \|_K = \inf \{ \lambda > 0 \, ; \, x \in \lambda K \}$.
However, as in the proof of Corollary \ref{cor_1024},
we see that
$$ \frac{1}{2} \leq \gamma_n(K) \leq \gamma_n \left( \frac{\sqrt{n}}{2}
B^n \right) + \gamma_n \left( K \setminus \frac{\sqrt{n}}{2} B^n \right)
\leq e^{-cn} + \sigma_{n-1} \left( \frac{2}{\sqrt{n}} K  \right). $$
Hence $\sigma_{n-1} \left( \frac{2}{\sqrt{n}} K  \right) \geq 1/2
- \exp(-cn)$. In other words, in a large subset of $S^{n-1}$,
the norm $\| x \|_K$ is at most $2 / \sqrt{n}$.
In \cite[Lemma 5.2.3]{AGM} it is explained how
concentration inequalities upgrade this
fact to the desired bound $M(K) \leq C / \sqrt{n}$.
\end{proof}

Our next observation is that the
assumption  $\gamma_n(K) \geq 1/2$ in Theorem \ref{thm_1405}
is too strong, and may be weakened to
the requirement that $\gamma_n(K) \geq \exp(-c t)$.

\begin{theo}
Let $1 \leq t \leq n$ and let $K \subseteq \RR^n$ be a
centrally-symmetric convex body with $\gamma_n(K) \geq \exp(-c_0 t)$.
Let $E \subseteq \RR^n$ be a random subspace of dimension
$n - t$. Then with probability of at least $1 - C \exp(-c t)$,
$$ c_1 \sqrt{t} B_E \subseteq Proj_E( K ). $$
\label{thm_1059}
\end{theo}

\begin{proof}
We may select the universal constant $c_0 > 0$ so that
the probability that a standard normal random variable
exceeds $\tilde{c} \sqrt{t}/2$, where $\tilde{c}$ is the constant
in the conclusion of Theorem \ref{thm_1405}, is at most $e^{-c_0t}$.

According to the Gaussian isoperimetric inequality,
for a half-space $H \subseteq \RR^n$,
$$ \gamma_n(K) = \gamma_n(H) \qquad \Longrightarrow
\qquad \gamma_n(K + (\tilde{c} \sqrt{t}/2) B^n)
\geq \gamma_n(H + (\tilde{c} \sqrt{t}/2) B^n). $$
Since $\gamma_n(H) = \gamma_n(K) \geq \exp(-c_0 t)$,
the choice of $c_0$ implies that the distance between the half-space $H$
and the origin is at most $\tilde{c} \sqrt{t} / 2$.
Consequently, $H + (\tilde{c} \sqrt{t}/2) B^n$
is a half-space containing the origin,
thus its Gaussian measure is at least $1/2$. Hence
$$ T := K + \frac{\tilde{c}}{2} \sqrt{t} B^n $$
is a centrally-symmetric convex body with $\gamma_n(T) \geq 1/2$.
By Theorem \ref{thm_1405}, with probability at least $1 - C \exp(-c
t)$
of selecting $E$,
\begin{equation}  \tilde{c} \sqrt{t} B_E \subseteq Proj_E(T)
= Proj_E(K) + Proj_E \left(\frac{\tilde{c} \sqrt{t} }{2}
B^n \right) = Proj_E(K) + \frac{\tilde{c} \sqrt{t} }{2} B_E.
\label{eq_1054}
\end{equation}
Since $B_E$ and $Proj_E(K)$ are convex, we deduce from
(\ref{eq_1054}) that $(\tilde{c} \sqrt{t} / 2) B_E
\subseteq Proj_E(K)$, completing the proof.
\end{proof}

\noindent
{\bf Remark.}
Consider the case where $$ K = [-r,r]^n $$ is an
$n$-dimensional cube, for $r = c \sqrt{\log(n/\ell)}$.
In this case one may easily verify that
$\gamma_n(K) \geq \exp(-c_0 \ell)$. Thus, according to
the last Theorem, with high probability
a random $(n-\ell)$-dimensional projection of $K$
contains a Euclidean ball of radius $\tilde{c} \sqrt{\ell}$.
This recovers an inequality by
Garnaev and Gluskin \cite{GG}. Moreover, the tightness
of the Garnaev-Gluskin result shows that the requirement that
$\gamma_n(K) \geq \exp(-c_0 \ell)$ in the Theorem is optimal.

\begin{coro} Let $K \subseteq \RR^n$ be a centrally-symmetric
convex body with $\gamma_n(K) \geq \exp(-c_0 t)$ with $1 \leq t \leq n$.
Then there exists a $t$-dimensional subspace
$E \subseteq \RR^n$ such that for any $v \in \RR^n$,
\begin{equation}  |v| \leq \sqrt{t} \qquad \Longrightarrow
\qquad E \cap (v + C K) \neq \emptyset. \label{eq_1100}
\end{equation} \label{cor_1104}
\end{coro}

\begin{proof} Write $F = E^{\perp}$. Condition (\ref{eq_1100})
is equivalent to $\sqrt{t} B_F \subseteq Proj_F(C K)$.
The corollary thus follows from Theorem \ref{thm_1059}
with $C=1/c_1$.
\end{proof}

\bigskip \begin{proof}[Proof of Theorem \ref{t13}]
We may assume that $t \leq n$ as otherwise the conclusion of the theorem is trivial.
We may also assume that $C>5/c_0$ where $c_0 > 0$ is  the universal constant from Corollary \ref{cor_1104}. That is,
$c_0 t \geq 5 \cdot \eps^{-2} \log(2 + \eps^2 n)$, thus
$$
\eps \geq  2 \sqrt{\frac{\log (2 +  n/(c_0 t))}{c_0 t}}.
$$
As in the proof of Theorem \ref{prop_948_}
identify $\RR^t$ with the subspace of $\RR^{n}$ of all vectors
whose last $n-t$ coordinates vanish,
thus we may write $\RR^t \subseteq \RR^{n}$.
Let $U \in O(n)$ be an orthogonal matrix to
be determined later on. For all $i,j$,
and for every vectors $X_i,Y_j$ in $\RR^n$
\begin{equation}
\left| \left \langle \frac{X_i}{\sqrt{t}}, \frac{Y_j}
{\sqrt{t}} \right \rangle - \left \langle a_i, b_j
\right \rangle \right | \leq
\left| \left \langle \frac{U X_i}{{\sqrt{t}}} - a_i, b_j
\right \rangle \right | + \left| \left \langle \frac{X_i}{\sqrt{t}},
\frac{Y_j}{\sqrt{t}} - U^{-1} b_j \right \rangle \right |.
\label{eq_9361}
\end{equation}
We next bound  the first summand on the
right-hand side of (\ref{eq_9361}). (We will later observe that we
can ensure that the second summand vanishes).
Define
$$ K = \left \{ x \in \RR^{n} \, ; \, \left| \left \langle
\frac{x}{\sqrt{t}}, b_j \right \rangle \right|
\leq \sqrt{c_0} \eps \, \ \text{for} \ j=1,\ldots,n \right \},  $$
where $c_0 > 0$ is still the constant from Corollary \ref{cor_1104}.
By the Khatri-Sidak lemma
\begin{align*} \gamma_{n}(K) & \geq \prod_{j=1}^n
\gamma_{n} \left( \left \{ x \in \RR^{n} \, ; \, \left|
\left \langle \frac{x}{\sqrt{t}}, b_j \right \rangle \right|
\leq \sqrt{c_0} \eps \right \} \right)
= \prod_{j=1}^{n} \left( 1 - 2 \Phi(\sqrt{c_0 t} \eps / |b_j|) \right)
\\ &
 \geq \left(1 - 2 \Phi \left( 2 \sqrt{\log (2+n/ (c_0t))} \right) \right)^{n}
\geq \left(1 - \frac{c_0 t}{n+ c_0 t} \right)^n \geq e^{-c_0 t}. \nonumber
\end{align*}
By Corollary \ref{cor_1104}
there exists a $t$-dimensional
subspace $E \subseteq \RR^{n}$ such that for any $v \in \RR^{n}$
$$ |v| \leq \sqrt{t} \qquad \Longrightarrow \qquad
E \cap (v +CK) \neq \emptyset. $$
Let us now set $U \in O(n)$ to be any orthogonal
transformation with $U(\RR^t) = E$, and choose $Ux_i \in E$ so that
$Ux_i -\sqrt t a_i \in CK$. Finally define
$y_j = \sqrt{t} P(U^{-1} b_j)$, where $P( z_1, z_2, \ldots ,z_n)=
(z_1,z_2, \ldots ,z_t)$.

This gives an upper bound of $C \sqrt{c_0}\eps$
for the right-hand side of (\ref{eq_9361}) for all $i,j$,
implying a variant of Theorem \ref{t13}
in which $\eps$ is replaced by $C \sqrt{c_0}\eps$.
By adjusting the constants, this variant is
equivalent to the original formulation, completing the proof.
\end{proof}

Note that the proof of Theorem \ref{t13} leads to a randomized, polynomial-time algorithm
for the computation of the $x_i, y_j$. Indeed, the orthogonal matrix $U \in O(n)$ can be chosen
randomly, and according to Theorem \ref{thm_1059} and Corollary \ref{cor_1104} such a random  matrix
works with probability of at least $1 - C \exp(-c t)$. Once the matrix $U$ is known,
the computation of $x_i$ such that $U x_i \in E$ and $U x_i \in \sqrt{t} a_i + C K$
may be done by linear programming. The computation of the $y_j$ is even quicker,
since we set $y_j = \sqrt{t} P(U^{-1} b_j)$. The total running time of the algorithm is clearly
polynomial in the input size.

\section{Concluding remarks}
\begin{itemize}
\item
By the first two parts  of Theorem \ref{t11},
$f(n,n,2 \eps)$ is much bigger than $f(n,k,\eps)$ for any
$k < c \frac{\log n}{\eps^2}$ for some absolute constant $c>0$,
implying that, as proved
recently by Larsen and Nelson \cite{LN}, the
$\frac{\log n}{\eps^2}$ bound in the Johnson-Lindenstrauss Lemma
\cite{JL} is tight.  The first part of Corollary \ref{c12}
follows by a similar reasoning. It can also be derived directly
from the result for $k=\log n/\eps^2$.
As for the ``Moreover'' part, it follows
by combining the Johnson-Lindenstrauss Lemma
with the lower bound of Theorem \ref{t11}. Corollary \ref{c24}
follows from Theorem \ref{t23} using essentially the same argument.
\item
It is worth noting that in
the proof of Theorem \ref{t31} the inner product of each rounded
vector with itself is typically not close to the square of its
original norm and hence it is crucial to keep the approximate norms
separately. An alternative, less natural possibility is to
store two independent rounded copies of each vector and use their
inner product as an approximation for its norm. This, of course,
doubles the  length of the sketch and there is no reason to do it.
For the same reason in the proof of Theorem \ref{t11} in Section 2
we had to handle norms separately and consider only inner products
between distinct vectors. Indeed, in this proof after the
conditioning $V_i$ is likely to
have much bigger norm than $w_i$, and yet the inner products of
distinct $V_i,V_j$ are typically very close to those
of the corresponding distinct $w_i,w_j$.
\item
The problem of maintaining all square distances between the points up to a
relative  error of $\eps$ is more difficult than the one
considered here. Our lower bounds, of course, hold, see \cite{IW}
for the best known upper bounds. For this problem there is still
a logarithmic
gap between the upper and lower bounds.
\item
The assertion  of Theorem \ref{thm_944}  for $m=2n$ and
$\eps=\frac{C}{\sqrt {n}}$  is tight up to a constant factor
even for the case that $a_i=b_i$ for all $i$  and the vectors
$a_i$ form an orthonormal basis of $\RR^{2n}$. Indeed, it is well
known (see, e.g., \cite{Al}) that any $2n$ by $2n$ matrix in which
every entry differs from the corresponding entry of the identity
matrix of dimension $2n$ by less than, say, $\frac{1}{2 \sqrt n}$
has rank exceeding $n$.
\item
For a matrix $A$, the $\gamma_2$-norm of $A$ denoted by
$\gamma_2(A)$ is the minimum possible value, over all
factorizations $A=XY$, of the product of the
maximum $\ell_2$-norm of a row of $X$
and the maximum $\ell_2$-norm of a column of $Y$. Therefore,
an equivalent formulation of the statement of Theorem
\ref{t13} for $\eps=O(1/\sqrt n)$ is that for any
$n$ by $n$ matrix $A$ satisfying $\gamma_2(A) \leq 1$ there
is an $n$ by $n$ matrix $B$ of rank at most, say, $n/10$ so that
$|A_{ij}-B_{ij}| \leq O(1/\sqrt n)$ for all $i,j$. It is worth
noting that the assumption that $\gamma_2(A) \leq 1$ here is
essential and cannot be replaced by a similar bound on
$\max |A_{ij}|$. Indeed, it is known (see \cite{ALSV}, Theorem 1.2)
that if $A$ is an $n$ by $n$ Hadamard matrix then any $B$ as
above has rank at least $n-O(1)$.
\item
Conjecture \ref{c14} remains open, it seems tempting to try to
iterate the assertion of Theorem \ref{thm_944} in order to prove
it. This does not work as the norms of the vectors $x_i$
and $y_i$ obtained in the proof may be much larger than $1$ (while
bounded), causing the errors in the iteration process to grow too
much. An equivalent formulation of this fact is that the
$\gamma_2$-norm
of the matrix $\langle a_i, b_j
\rangle$ is $1$ whereas that of its approximating lower rank matrix
is a larger constant.
\end{itemize}
\vspace{0.2cm}

\noindent
{\bf Acknowledgment}\, We thank Jaroslaw Blasiok,
Kasper Green Larsen and especially Jelani
Nelson for helpful comments, and for noting the
relation to the paper \cite{KOR}.

\end{document}